\documentclass[10pt]{amsart}

\usepackage[latin1]{inputenc}
\usepackage{amsmath}
\usepackage{amsfonts}
\usepackage{amssymb}
\usepackage{ytableau}
\usepackage[mathscr]{eucal}
\usepackage{diagrams}
\usepackage{xy,color,graphicx}
\usepackage{hyperref}
\xyoption{all}
\usepackage{array}
\usepackage{enumerate}
\usepackage{url}

\newcommand{\vtx}[1]{*+[o][F-]{\scriptscriptstyle #1}}

\newtheorem{definition}{Definition}

\newtheorem{lemma}{Lemma}

\title{What is a noncommutative topos?}
\author{Karin Cvetko-Vah} \thanks{Karin Cvetko-Vah acknowledges the financial support from the Slovenian Research Agency (research core funding No. P1-0222)}
\address{University of Ljubljana, Faculty of Mathematics and Physics\\Jadranska 19\\1000 Ljubljana (Slovenia) \\ {\tt karin.cvetko@fmf.uni-lj.si}}
\author{Jens Hemelaer}\thanks{Jens Hemelaer is a Ph.D. fellow of the Research Foundation - Flanders (FWO)}
\address{Department of Mathematics, University of Antwerp \\ 
 Middelheimlaan 1, B-2020 Antwerp (Belgium) \\ {\tt jens.hemelaer@uantwerpen.be}}
\author{Lieven Le Bruyn} 
\address{Department of Mathematics, University of Antwerp \\ 
 Middelheimlaan 1, B-2020 Antwerp (Belgium) \\ {\tt lieven.lebruyn@uantwerpen.be}}

\begin{document}
\sloppy

\maketitle

\begin{abstract} In \cite{KarinNCH} noncommutative frames were introduced, generalizing the usual notion of frames of open sets of a topological space. In this paper we extend this notion to noncommutative versions of Grothendieck topologies and their associated noncommutative toposes of sheaves of sets.
\end{abstract}

\vskip 3mm

\begin{center}
{\it For Fred Van Oystaeyen on his 70th birthday.}
\end{center}

\section{Introduction}

The set $\Omega$ of all open sets of a topological space $X$ is a complete Heyting algebra: it is partially ordered under inclusion, the join $\vee$ and meet $\wedge$ operations  are resp.\ union and intersection of opens, the implication operator $U \rightarrow V$ is defined to be the largest open set $W$ satisfying $W \cap U \subseteq V$, and it has a unique bottom element $0 = \emptyset$ and top element $1 = X$, see for example \cite[\S I.8]{MM}.

Let $\mathcal{F}$ be a sheaf of sets over the constructible topology on $X$, that is the topology generated by all open {\em and} all closed subsets of $X$. For every open set $U$ in $X$ we consider $\{ (U,s)~|~s \in \Gamma(U,\mathcal{F}) \}$. The set $H$ of all such possible $(U,s)$ is partially ordered under $(U,s) \leq (V,t)$ if and only if $U \subseteq V$ and $t | U = s$. Fix a distinguished global section $g \in \Gamma(X,\mathcal{F})$.
We now define noncommutative operations of $H$ as follows
\begin{itemize}
\item{$(U,s) \wedge (V,t) = (U \cap V,s | U \cap V)$,}
\item{$(U,s) \vee (V,t) = (U \cup V,t \cup s | U-V)$,}
\item{$(U,s) \rightarrow (V,t) = (U \rightarrow V, t \cup g | (U \rightarrow V)-V)$}
\end{itemize}
$H$ still has a unique bottom element corresponding to $0 = \emptyset$, but now has a family $\{ (X,t)~|~t \in \Gamma(X,\mathcal{F}) \}$ of top elements, and observe that the downset of each of them $(X,t)_{\downarrow}$ is isomorphic to the Heyting algebra $\Omega$, and if we consider Green's equivalence relation $\mathcal{D}$
\[
(U,s)~\mathcal{D} ~(V,t) \quad \text{if and only if} \quad \begin{cases} (U,s) \wedge (V,t) \wedge (U,s) = (U,s) \\
(V,t) \wedge (U,s) \wedge (V,t) = (V,t) \end{cases} \]
then the equivalence classes $H / \mathcal{D}$ with the induced structures are isomorphic to $\Omega$ as Heyting algebras. $H$ is an example of a noncommutative complete Heyting algebra as introduced and studied in \cite{KarinNCH}. We can view $H$ as the set of opens of a noncommutative topological space with commutative shadow $X$.

\vskip 3mm

In this paper we aim to define, in a similar way, noncommutative counterparts of toposes $\mathbf{Sh}(\mathbf{C},J)$ of sheaves of sets with respect to a Grothendieck topology $J$ on a small category $\mathbf{C}$. Fred Van Oystaeyen suggested in his book 'Virtual topology and functor geometry' a possible approach:

\vskip 3mm

\begin{quote}
"One easily finds that the first main problem is to circumvent the notion of subobject classifier. An approach may be to allow a {\em family} of 'subobject classifiers' defined in a suitable way."  \cite[p. 44]{FVOVT}
\end{quote}

\vskip 3mm

Let $\widehat{\mathbf{C}}$ be the topos of presheaves on $\mathbf{C}$, that is, with objects all contravariant functors $\mathbf{P} : \mathbf{C} \rTo \mathbf{Sets}$ and with morphisms all natural transformations. Recall from \cite[\S III.7]{MM} that the natural transformation $true : \mathbf{1} \rTo \mathbf{\Omega}$ is the subobject classifier of $\widehat{\mathbf{C}}$, where for every object $C$ of $\mathbf{C}$ we take $\mathbf{\Omega}(C)$ to be the set of all sieves on $C$ and where the global section $true$ picks out the unique maximal sieve $\mathbf{y}(C)$ of all morphisms with codomain $C$. Each $\mathbf{\Omega}(C)$ is a complete Heyting algebra, that is, $\mathbf{\Omega}$ is a presheaf of complete Heyting algebras on $\mathbf{C}$. We will define a {\em noncommutative subobject classifier} $\mathbf{H}$ to be a presheaf of noncommutative complete Heyting algebras making the diagram below commute
\[
\xymatrix{\mathbf{C} \ar[rr]^{\mathbf{\Omega}} \ar[rd]_{\mathbf{H}} & & \mathbf{cHA} \\ & \mathbf{ncHA} \ar[ru]_{./\mathcal{D}} &} \]
where $./\mathcal{D} : \mathbf{ncHA} \rTo \mathbf{cHA}$ is the covariant functor sending a noncommutative complete Heyting algebra $H$ to its commutative shadow $H/\mathcal{D}$. Note that $\mathbf{H}$ has a subobject $t_{\mathbf{H}} : \mathbf{T} \rTo \mathbf{H}$ where $\mathbf{T}$ is the presheaf of top elements of $\mathbf{H}$. We will often recite these two mantras:

(1) : Occurrences of the terminal object $\mathbf{1}$ and $\mathbf{\Omega}$ in classical definitions should be replaced by the presheaves $\mathbf{T}$ and $\mathbf{H}$.

(2) : All noncommutative structures will determine {\em families} of classical structures, parametrized by the global sections of $\mathbf{T}$.

\vskip 3mm

Let us illustrate this in the definition of the noncommutative Heyting algebra $\mathbf{Sub}_{\mathbf{H}}(\mathbf{P})$ generalizing the classical Heyting algebra of subobjects $\mathbf{Sub}(\mathbf{P})$ of $\mathbf{P} \in \widehat{\mathbf{C}}$. Subobjects of $\mathbf{P}$ are in one-to-one correspondence with natural transformations $N : \mathbf{P} \rTo \mathbf{\Omega}$ via the pullback diagram on the left below
\[
\xymatrix{\mathbf{Q} \ar[r]^N \ar[d] & \mathbf{1} \ar[d]^{true} \\ \mathbf{P} \ar[r]_N & \mathbf{\Omega}} \hskip 3cm
\xymatrix{\mathbf{Q} \ar[r]^N \ar[d] & \mathbf{T} \ar[d]^{t_{\mathbf{H}}} \\ \mathbf{P} \ar[r]_N & \mathbf{H}} 
\]
Similarly, elements of $\mathbf{Sub}_{\mathbf{H}}(\mathbf{P})$ will be pairs $(\mathbf{Q},N)$ where $N : \mathbf{P} \rTo \mathbf{H}$ is a natural transformation and $\mathbf{Q}$ is the pullback subobject of the diagram on the right above. Because $\mathbf{H}$ is a presheaf of noncommutative Heyting algebras we have that if $N$ and $N'$ are natural transformations from $\mathbf{P}$ to $\mathbf{H}$ then so are $N \wedge N'$, $N \vee N'$ and $N \rightarrow N'$ as defined in lemma~\ref{structures}. This then allows us to define operations on $\mathbf{Sub}_{\mathbf{H}}(\mathbf{P})$
\[
\begin{cases}
(\mathbf{Q},N) \wedge (\mathbf{Q'},N') = (\mathbf{Q} \wedge \mathbf{Q'},N \wedge N') \\
(\mathbf{Q},N) \vee (\mathbf{Q'},N') = (\mathbf{Q} \vee \mathbf{Q'},N \vee N') \\
(\mathbf{Q},N) \rightarrow (\mathbf{Q'},N') = (\mathbf{Q} \rightarrow \mathbf{Q'},N \rightarrow N')
\end{cases}
\]
where we have the pull-back diagrams
\[
\xymatrix{\mathbf{Q}\wedge \mathbf{Q'} \ar[rr] \ar[d] & & \mathbf{T} \ar[d]^{t_{\mathbf{H}}} \\ \mathbf{P} \ar[rr]_{N \wedge N'} & & \mathbf{H}}  \quad
\xymatrix{\mathbf{Q} \vee \mathbf{Q'} \ar[rr] \ar[d] & & \mathbf{T} \ar[d]^{t_{\mathbf{H}}} \\ \mathbf{P} \ar[rr]_{N \vee N'} & & \mathbf{H}}  \quad
\xymatrix{\mathbf{Q} \rightarrow \mathbf{Q'} \ar[rr] \ar[d] & & \mathbf{T} \ar[d]^{t_{\mathbf{H}}} \\ \mathbf{P} \ar[rr]_{N \rightarrow N'} & & \mathbf{H}} 
\]
defining a noncommutative Heyting algebra structure. Let $\Gamma(\mathbf{T})$ be the set of global sections $g : \mathbf{1} \rTo \mathbf{T}$ of the presheaf of top elements $\mathbf{T}$, then there is a morphism
\[
\mathbf{sub}_{\mathbf{H}}(\mathbf{P}) \rTo \prod_{g \in \Gamma(\mathbf{T})} \mathbf{Sub}(\mathbf{P}) \qquad (\mathbf{Q},N) \mapsto (\mathbf{Q}_g)_{g \in \Gamma(\mathbf{T})}  \]
with $\mathbf{Q}_g$ determined by the diagram below
\[
\xymatrix{
\mathbf{Q}_g \ar[r] \ar[d] & \mathbf{1} \ar[d]^g \ar[rd]^{id} & \\
\mathbf{Q} \ar[r]^N \ar[d] & \mathbf{T} \ar[r] \ar[d]^{t_{\mathbf{H}}} & \mathbf{1} \ar[d]^{true} \\
\mathbf{P} \ar[r]_N & \mathbf{H} \ar[r]_{./\mathcal{D}} & \mathbf{\Omega}} \]
Having defined noncommutative subobject classifiers $\mathbf{H}$, we approach defining noncommutative Grothendieck topologies via generalizing Lawvere-Tierney topologies on $\widehat{\mathbf{C}}$, see for example \cite[\S V.1]{MM}. A {\em noncommutative Lawvere topology} will then be a natural transformation $j_{\mathbf{H}} : \mathbf{H} \rTo \mathbf{H}$ satisfying

(NLT1) : $j_{\mathbf{H}} \circ t_{\mathbf{H}} = t_{\mathbf{H}}$,

(NLT2) : $j_{\mathbf{H}} \circ j_{\mathbf{H}} = j_{\mathbf{H}}$,

\[
\xymatrix{\mathbf{T} \ar[r]^{t_{\mathbf{H}}} \ar[rd]_{t_{\mathbf{H}}} & \mathbf{H} \ar[d]^{j_{\mathbf{H}}} \\ & \mathbf{H}}~\qquad
\xymatrix{\mathbf{H} \ar[r]^{j_{\mathbf{H}}} \ar[rd]_{j_{\mathbf{H}}} & \mathbf{H} \ar[d]^{j_{\mathbf{H}}} \\ & \mathbf{H}}~\qquad
\]

(NLT3) : For every object $C$ in $\mathbf{C}$, every top-element $t \in \mathbf{T}(C)$ and all $x,y \in t_{\downarrow} \subset \mathbf{T}(C)$ we have
 the condition
\[
j_{\mathbf{H}}(C)(x \wedge y) = j_{\mathbf{H}}(C)(x) \wedge j_{\mathbf{H}}(C)(y) \]
Again, every global section $g : \mathbf{1} \rTo \mathbf{T}$ determines a Lawvere-Tierney topology on $\widehat{\mathbf{C}}$
via the restriction of $j_{\mathbf{H}}$ on $g_{\downarrow} \simeq \mathbf{\Omega}$.

As $\mathbf{C}$ is a small category there is a one-to-one correspondence between Lawvere-Tierney topologies on $\widehat{\mathbf{C}}$ and Grothendieck topologies on $\mathbf{C}$. Extending this, we have that a noncommutative Lawvere topology determines a {\em noncommutative Grothendieck topology} by associating to every object $C$ the following collection of elements from $\mathbf{Sub}_{\mathbf{H}}(\mathbf{y}C)$
\[
J_{\mathbf{H}}(C) = \{ (S,x) \in \mathbf{\Omega}(C) \times \mathbf{H}(C)~|~(S,x) \in \mathbf{Sub}_{\mathbf{H}}(\mathbf{y}C)~\text{and}~ j_{\mathbf{H}}(C)(x) \in \mathbf{T}(C) \} \]
This then allows us to define a presheaf $\mathbf{F}$ in the slice category $\widehat{\mathbf{C}}/\mathbf{T}$ to be a {\em sheaf} for the noncommutative Grothendieck topology $J_{\mathbf{H}}$ if and only if for every object $C$ of $\mathbf{C}$, every element $(S,x) \in J_{\mathbf{H}}(C)$, and every morphism $g$ in $\widehat{\mathbf{C}}/\mathbf{T}$
\[
\xymatrix{ & \mathbf{y} C \ar@{.>}[rd]^{\exists !} & \\ S \ar[rr]^g \ar[rd]_x \ar[ru] & & \mathbf{F} \ar[ld]^{\pi_{\mathbf{F}}} \\
& \mathbf{T} & } \]
there is a unique morphism $\mathbf{y} C \rTo \mathbf{F}$ in $\widehat{\mathbf{C}}$. Here $S \rTo^x \mathbf{T}$ is the pull-back map induced by the natural transformation $x : \mathbf{y} C \rTo \mathbf{H}$. The category of all such sheaves $\mathbf{Sh}(\mathbf{C},J_{\mathbf{H}})$ is then called a {\em noncommutative topos}.

In the last section we present a large class of examples of noncommutative subobject classifiers and give an explicit example of a noncommutative topos which is {\em not} a Grothendieck topos, nor even an elementary topos.

\section{Noncommutative Heyting algebras}

In this section we will recall the main structural results on noncommutative (complete) Heyting algebras obtained in \cite{KarinNCH}.

\vskip 2mm

Recall that a bounded lattice $L$ is a set with two distinguished elements $0$ and $1$ and two binary operations $\vee$ and $\wedge$ which are both idempotent, associative and commutative and satisfy the identities
\[
1 \wedge x = x, \qquad 0 \vee x = x \]
\[
x \wedge(y \vee x) = x = (x \wedge y) \vee x \]
$L$ is said to be distributive if we have the added identity
\[
x \wedge(y \vee z) = (x \wedge y) \vee (x \wedge z) \]
A {\em Heyting algebra} $H$ is a bounded distributive lattice $(H,0,1,\vee,\wedge)$ which is also a partially ordered set under $\leq$ and has a binary operation $\rightarrow$ satisfying the following set of axioms

(H1): $(x \rightarrow x) = 1$,

(H2): $x \wedge (x \rightarrow y) = x \wedge y$,

(H3): $y \wedge (x \rightarrow y) = y$,

(H4): $x \rightarrow (y \wedge z) = (x \rightarrow y) \wedge (x \rightarrow z)$.

\noindent
Equivalently, these axioms can be replaced by the following single axiom

(HA): $x \wedge y \leq z$ iff $x \leq y \rightarrow z$.

A Heyting algebra $H$ is said to be complete if every subset $\{ x_i : i \in I \}$ of $H$ has a supremum $\bigvee_i x_i$ and an infimum $\bigwedge_i x_i$, satisfying the infinite distributive law $\bigvee_i (y \wedge x_i) = y \wedge \bigvee_i x_i$. With $\mathbf{cHA}$ we denote the category of all join-complete Heyting algebras with morphisms the lattice, order preserving maps, preserving $0$ and $1$.

\vskip 5mm

 In \cite{KarinNCH} noncommutative Heyting algebras were introduced and studied.  A {\em skew lattice} is an algebra $(L,\wedge,\vee)$ where $\wedge$ and $\vee$ are idempotent and associative binary operations satisfying the identities
 \[
 x \wedge (x \vee y) = x = x \vee (x \wedge y) \quad \text{and} \quad (x \wedge y) \vee y = y = (x \vee y) \wedge y \]
 A skew lattice is {\em strongly distributive} if it satisfies the additional identities
 \[
 (x \vee y) \wedge z = (x \wedge z) \vee (y \wedge z) \quad \text{and} \quad x \wedge(y \vee z) = (x \wedge y) \vee (x \wedge z) \]

 Green's equivalence relation $\mathcal{D}$ on a skew lattice is defined via $x\, \mathcal D\, y$ iff $x \wedge y \wedge x = x$ and $y \wedge x \wedge y = y$. We will denote the $\mathcal{D}$-equivalence class of $x \in L$ by $\mathcal{D}_x$. The set of equivalence classes $L / \mathcal{D}$ with the induced operations is a distributive lattice and if $L / \mathcal{D}$ has a maximal element $1$ we call the corresponding $\mathcal{D}$-class in $L$ the set of {\em top elements} and denote it with $T$.

 A skew lattice has a natural partial order defined by $x \leq y$ iff $x \wedge y = x = y \wedge x$. With $x_{\downarrow}$ we will denote the subset consisting of all $y \in L$ such that $y \leq x$. 
By a result of Leech \cite{Leech}, $x_{\downarrow}$ is a distributive lattice for any $x$ in a strongly distributive skew lattice$S$. If  $S$ has a maximal element $1$ then $S=1_\downarrow$, which implies that $S$ is necessarily commutative. That is, we have to sacrifice a unique top element when passing to the noncommutative setting.
 
 \vskip 5mm

 From \cite[\S 3]{KarinNCH} we recall that a {\em noncommutative Heyting algebra} is an algebra $(H,\wedge,\vee,0,t)$ where $(H,\wedge,\vee,0)$ is a strongly distributive lattice with bottom $0$ and a top $\mathcal{D}$-class $T$, $t$ is a distinguished element of $T$ and $\rightarrow$ is a binary operation satisfying the following conditions
\begin{itemize}
\item[(NH1)] $x\rightarrow y = (y\lor(t\land  x\land t)\lor y)\rightarrow y$,
\item[(NH2)] $x\rightarrow x = x\lor t\lor x$,
\item[(NH3)]  $x\land (x\rightarrow y) \land x = x\land y \land x$, 
  \item[(NH4)] $y\land (x\rightarrow y)=y$ and  $(x\rightarrow y)\land y=y$,
  \item[(NH5)] $x \rightarrow (t\land (y\land z)\land t)=
  ( x\rightarrow (t\land y\land t)) \land (x \rightarrow (t \land  z\land t))$.
\end{itemize}

The main structural result on noncommutative Heyting algebras, \cite[Thm.\ 3.5]{KarinNCH} asserts that if $(H,\wedge,\vee,\rightarrow,0,t)$ is a noncommutative Heyting algebra, then
\begin{enumerate}
\item{$(t_{\downarrow},\wedge,\vee,\rightarrow,0,t)$ is a Heyting algebra with a unique top element $t$, isomorphic to $H / \mathcal{D}$.}
\item{For any $t' \in T$ also $(t'_{\downarrow},\wedge,\vee,\rightarrow,0,t')$ is a Heyting algebra and the map
\[
\phi~:~t_{\downarrow} \rTo t'_{\downarrow} \qquad x \mapsto t' \wedge x \wedge t' \]
is an isomorphism of Heyting algebras and for all $x \in t_{\downarrow}$ we have $x \,\mathcal D\, \phi(x)$.}
\end{enumerate}

From now on we will assume that the noncommutative Heyting algebra is {\em complete}, that is if all {\em commuting} subsets have suprema and infima in their partial ordering, and they satisfy the infinite distributive laws
\[
(\bigvee_i x_i) \wedge y = \bigvee_i (x_i \wedge y) \quad \text{and} \quad x \wedge (\bigvee_i y_i) = \bigvee_i (x \wedge y_i) \]
for all $x,y \in H$ and all commuting subsets $(x_i)_i$ and $(y_i)_i$.

With $\mathbf{ncHA}$ we denote the category with objects all complete noncommutative Heyting algebras and maps preserving $\leq$, $\wedge$, $\vee$, $\rightarrow$, $0$ and the distinguished top element $t$.

From \cite[Thm.\ 3.5.(iii)]{KarinNCH} we recall that Green's relation $\mathcal{D}$ is a congruence on a noncommutative Heyting algebra $H$ and that the Heyting algebra $H/\mathcal{D}$ is its maximal lattice image, that is, every noncommutative Heyting algebra morphism $H \rTo H_c$ to a (commutative) Heyting algebra $H_c$ factors through the quotient $\pi_{\mathcal{D}} : H \rOnto H/\mathcal{D}$. We can rephrase this as

\begin{lemma} Green's relation $\mathcal{D}$ induces a covariant functor
\[
/\mathcal{D}~:~\mathbf{ncHA} \rTo \mathbf{cHA} \qquad H \mapsto H/\mathcal{D} \]
\end{lemma}

\vskip 3mm

\section{Noncommutative subobject classifiers}

Let $\mathbf{C}$ be a small category and $\mathbf{P}$ a presheaf on $\mathbf{C}$, that is, a contravariant functor $\mathbf{P} : \mathbf{C} \rTo \mathbf{Sets}$. We recall that subobjects of $\mathbf{P}$ correspond to natural transformations $N : \mathbf{P} \rTo \mathbf{\Omega}$ to the subobject classifier $\mathbf{\Omega}$, which is a presheaf of complete Heyting algebras on $\mathbf{C}$.

Motivated by this, we will consider the set $(\mathbf{P},\mathbf{H})$ of all natural transformations $N : \mathbf{P} \rTo \mathbf{H}$ to a presheaf $\mathbf{H}$ of noncommutative complete Heyting algebras on $\mathbf{C}$ and equip this set with a noncommutative Heyting algebra structure.

\vskip 3mm

Let $C$ be an object of $\mathbf{C}$. A sieve $S$ on $C$ is a set of morphisms in $\mathbf{C}$, all with codomain $C$, such that if $g \in S$ then $h \circ g \in S$ whenever this composition makes sense. With $\mathbf{\Omega}(C)$ we will denote the set of all sieves on $C$. If $S$ is a sieve on $C$ and $h : D \rTo C$ a morphism in $\mathbf{C}$, then
\[
h^*(S) = \{ g~|~codom(g)=D, h \circ g \in S \} \]
is a sieve on $D$. Hence, $\mathbf{\Omega}$ is a contravariant functor $\mathbf{\Omega} : \mathbf{C} \rTo \mathbf{Sets}$, that is, a presheaf on $\mathbf{C}$. In fact, as unions and intersections of sieves on $C$ are again sieves on $C$, each $\mathbf{\Omega}(C)$ is a complete Heyting algebra with bottom element $0 = \emptyset$ and unique maximal element $1 = \mathbf{y}(C)$ the set of all morphisms with codomain $C$. Moreover, for any $h : D \rTo C$ we have that $h^* : \mathbf{\Omega}(C) \rTo \mathbf{\Omega}(D)$ is a morphism of Heyting algebras. That is, we have a contravariant functor
\[
\mathbf{\Omega} : \mathbf{C} \rTo \mathbf{cHA} \]
to the category $\mathbf{cHA}$ of complete Heyting algebras. Assigning to each $C$ the maximal element $1 = \mathbf{y}(C)$ defines a global section of $\mathbf{\Omega}$
\[
true~:~\mathbf{1} \rTo \mathbf{\Omega} \]
which is the subobject classifier in $\widehat{\mathbf{C}}$, the topos of all presheaves of sets on $\mathbf{C}$., see \cite[p. 37-39]{MM}. That is, for every presheaf $\mathbf{P} \in \widehat{\mathbf{C}}$ there is a natural one-to-one correspondence between natural transformations $N : \mathbf{P} \rTo \mathbf{\Omega}$ and subobjects $\mathbf{Q}$ of $\mathbf{P}$ in $\widehat{\mathbf{C}}$, given by the pullback diagram
\[
\xymatrix{\mathbf{Q} \ar[r] \ar[d] & \mathbf{1} \ar[d]^{true} \\
\mathbf{P} \ar[r]_N & \mathbf{\Omega}} \]

\vskip 5mm

With this in mind, let us start with a presheaf $\mathbf{H}$ of noncommutative complete Heyting algebras on $\mathbf{C}$, that is, a contravariant functor
\[
\mathbf{H}~:~\mathbf{C} \rTo \mathbf{ncHA} \]
Every morphism $D \rTo^f C$ in $\mathbf{C}$ induces a morphism of noncommutative complete Heyting algebras
\[
H(f) : \mathbf{H}(C) \rTo \mathbf{H}(D) \]
and, in particular, it induces a map on the sets of top elements of these noncommutative Heyting algebras
\[
\mathbf{T}(f)~:~\mathbf{T}(C) = T(\mathbf{H}(C)) \rTo^{\mathbf{H}(f)} T(\mathbf{H}(D)) = \mathbf{T}(D) \]
That is, taking for every object $C$ in $\mathbf{C}$ the set of top elements $\mathbf{T}(C)$ of the noncommutative complete Heyting algebra $\mathbf{H}(C)$ is a presheaf of sets on $\mathbf{C}$, and the inclusions $\mathbf{T}(C) \subseteq \mathbf{H}(C)$ define a natural transformation 
\[
t_{\mathbf{H}}~:~\mathbf{T} \rTo \mathbf{H} \]

\begin{lemma} Let $\mathbf{P} \in \widehat{\mathbf{C}}$ and let $N,N' : \mathbf{P} \rTo \mathbf{H}$ be natural transformations, then the maps
\[
\begin{cases}
(N \wedge N')(C) : \mathbf{P}(C) \rTo \mathbf{H}(C) \qquad x \mapsto N(C)(x) \wedge N'(C)(x) \\
(N \vee N')(C) : \mathbf{P}(C) \rTo \mathbf{H}(C) \qquad x \mapsto N(C)(x) \vee N'(C)(x) \\
(N \rightarrow N')(C) : \mathbf{P}(C) \rTo \mathbf{H}(C) \qquad x \mapsto N(C)(x) \rightarrow N'(C)(x)
\end{cases}
\]
define natural transformation $N \wedge N', N \vee N', N \rightarrow N' : \mathbf{P} \rTo \mathbf{H}$.
\end{lemma}

\begin{proof} For every morphism $D \rTo^f C$ in $\mathbf{C}$ we have to verify that the diagram below is commutative
\[
\xymatrix{\mathbf{P}(C) \ar[rr]^{(N \wedge N')(C)} \ar[d]_{\mathbf{P}(f)} & & \mathbf{H}(C) \ar[d]^{\mathbf{H}(f)} \\
\mathbf{P}(D) \ar[rr]_{(N \wedge N')(D)} & & \mathbf{H}(D)} \]
For every $x \in \mathbf{P}(C)$ we have that $\mathbf{H}(f)((N \wedge N')(C)(x)) =$
\[
\mathbf{H}(f)(N(C)(x) \wedge N'(C)(x)) = \mathbf{H}(f)(N(C)(x)) \wedge \mathbf{H}(f)(N'(C)(x)) \]
where the last equality follows from $\mathbf{H}(f)$ being a morphism of noncommutative complete Heyting algebras. Because $N$ and $N'$ are natural transformations, we have the equalities
\[
\mathbf{H}(f)(N(C)(x)) = N(D)(\mathbf{P}(f)(x)) \quad \text{and} \quad \mathbf{H}(f)(N'(C)(x)) = N'(D)(\mathbf{P}(f)(x)) \]
and so the term above is equal to
\[
N(D)(\mathbf{P}(f)(x)) \wedge N'(D)(\mathbf{P}(f)(x)) = (N \wedge N')(D)(\mathbf{P}(f)(x)) \]
The proofs for $N \vee N'$ and $N \rightarrow N'$ proceed similarly. 
\end{proof}

Every natural transformation $N : \mathbf{P} \rTo \mathbf{H}$ determines a pair $(\mathbf{Q},N)$ where  $\mathbf{Q}$ is a subobject of $\mathbf{P}$ via the pullback diagram
\[
\xymatrix{\mathbf{Q} \ar[r]^N \ar[d] & \mathbf{T} \ar[d]^{t_{\mathbf{T}}}  \\
\mathbf{P} \ar[r]_N & \mathbf{H} } \]
With $\mathbf{Sub}_{\mathbf{H}}(\mathbf{P})$ we denote the set of all such pairs $(\mathbf{Q},N)$ determined by a natural transformation $N : \mathbf{P} \rTo \mathbf{H}$.
 
 \begin{lemma}  \label{structures} On the poset $\mathbf{Sub}_{\mathbf{H}}(\mathbf{P})$ we can define operations
\[
\begin{cases}
(\mathbf{Q},N) \wedge (\mathbf{Q'},N') = (\mathbf{Q} \wedge \mathbf{Q'},N \wedge N') \\
(\mathbf{Q},N) \vee (\mathbf{Q'},N') = (\mathbf{Q} \vee \mathbf{Q'},N \vee N') \\
(\mathbf{Q},N) \rightarrow (\mathbf{Q'},N') = (\mathbf{Q} \rightarrow \mathbf{Q'},N \rightarrow N')
\end{cases}
\]
where we have the pull-back diagrams
\[
\xymatrix{\mathbf{Q}\wedge \mathbf{Q'} \ar[rr] \ar[d] & & \mathbf{T} \ar[d]^{t_{\mathbf{H}}} \\ \mathbf{P} \ar[rr]_{N \wedge N'} & & \mathbf{H}}  \quad
\xymatrix{\mathbf{Q} \vee \mathbf{Q'} \ar[rr] \ar[d] & & \mathbf{T} \ar[d]^{t_{\mathbf{H}}} \\ \mathbf{P} \ar[rr]_{N \vee N'} & & \mathbf{H}}  \quad
\xymatrix{\mathbf{Q} \rightarrow \mathbf{Q'} \ar[rr] \ar[d] & & \mathbf{T} \ar[d]^{t_{\mathbf{H}}} \\ \mathbf{P} \ar[rr]_{N \rightarrow N'} & & \mathbf{H}} 
\]
These operations turn the set $\mathbf{Sub}_{\mathbf{H}}(\mathbf{P})$ into a noncommutative complete Heyting algebra with minimal element $(\emptyset,N_0)$ and distinguished top element $(\mathbf{P},N_d)$, where the natural transformations $N_0,N_d : \mathbf{P} \rTo \mathbf{H}$ are the compositions 
\[
N_0 : \mathbf{P} \rTo \mathbf{1} \rTo^0 \mathbf{H} \quad \text{and} \quad N_d : \mathbf{P} \rTo \mathbf{1} \rTo^d \mathbf{H} \]
with the left-most morphism the unique map to the terminal object $\mathbf{1}$ and $d$ the global section of $\mathbf{H}$ determined by the distinguished elements. The top-elements are exactly the pairs $(\mathbf{P},N)$ where $N : \mathbf{P} \rTo \mathbf{T}$ is a natural transformation.
\end{lemma}

\begin{proof} Follows from the previous lemma and uniqueness of pull-backs.
\end{proof}

 \begin{definition} A presheaf $\mathbf{H}$ of noncommutative complete Heyting algebras on $\mathbf{C}$ is said to be a {\em noncommutative subobject classifier} if $\mathbf{H}/\mathcal{D} \simeq \mathbf{\Omega}$.
 \end{definition}
 
 \begin{lemma} If $\mathbf{H}$ is a noncommutative subobject classifier, then for every presheaf $\mathbf{P}$ on $\mathbf{C}$, we have a surjective  morphism of (noncommutative) complete Heyting algebras
 \[
\mathbf{Sub}_{\mathbf{H}}(\mathbf{P}) \rOnto \mathbf{Sub}(\mathbf{P}) \]
 \end{lemma}
 
 \begin{proof} The map is determined by sending a pair $(\mathbf{Q},N)$ to $\mathbf{Q}$. Or, equivalently, by composing with the quotient map of noncommutative complete Heyting algebras dividing out Green's relation
 \[
 \xymatrix{\mathbf{Q} \ar[r]^N \ar[d] & \mathbf{T} \ar[r] \ar[d]^{t_{\mathbf{H}}} & \mathbf{1} \ar[d] \\
 \mathbf{P} \ar[r]_N & \mathbf{H} \ar[r]_{./\mathcal{D}} & \mathbf{\Omega}} \]
 Let $d : \mathbf{1} \rTo \mathbf{H}$ be the global section corresponding to the distinguished top element, then the maps (of noncommutative complete Heyting algebras)
 \[
 \mathbf{\Omega}(C) \rTo^{\simeq} d(C)(1)_{\downarrow} \rInto \mathbf{H}(C) \]
 determine a natural transformation $\mathbf{\Omega} \rTo^i \mathbf{H}$. If $\mathbf{Q}$ is the subobject of $\mathbf{P}$ corresponding to the natural transformation $N : \mathbf{P} \rTo \mathbf{\Omega}$ then the composition $i \circ N$ is an element of $(\mathbf{P},\mathbf{H})$ mapping to $\mathbf{Q}$.
 \end{proof}

\section{Noncommutative Grothendieck topologies}

 In this section we will introduce noncommutative Grothendieck topologies and their corresponding toposes of sheaves. We will first extend the notion of Lawvere-Tierney topologies, which are certain closure operations on $\mathbf{\Omega}$, to noncommutative subobject classifiers. As Lawvere-Tierney topologies coincide with Grothendieck topologies when the category $\mathbf{C}$ is small, we will then determine the corresponding noncommutative Grothendieck topologies and define sheaves over them.
 
 \vskip 3mm
 
 A {\em Lawvere-Tierney topology} on $\widehat{\mathbf{C}}$, see for example \cite[V.\S 1]{MM}, is a natural transformation $j : \mathbf{\Omega} \rTo \mathbf{\Omega}$ satisfying the following three properties

(LT1): $j \circ \text{true} = \text{true}$;

(LT2): $j \circ j = j$;

(LT3): $j \circ \wedge = \wedge \circ (j \times j)$.

\[
\xymatrix{\mathbf{1} \ar[r]^{true} \ar[rd]_{true} & \mathbf{\Omega} \ar[d]^{j} \\ & \mathbf{\Omega}}~\qquad
\xymatrix{\mathbf{\Omega} \ar[r]^{j} \ar[rd]_{j} & \mathbf{\Omega} \ar[d]^{j} \\ & \mathbf{\Omega}}~\qquad
\xymatrix{\mathbf{\Omega} \times \mathbf{\Omega} \ar[r]^{\wedge} \ar[d]_{j \times j} & \mathbf{\Omega} \ar[d]^{j} \\ \mathbf{\Omega} \times \mathbf{\Omega} \ar[r]_{\wedge} & \mathbf{\Omega}}  \]
 
 Motivated by this we define, for a noncommutative subobject classifier $\mathbf{H}$ with presheaf of top-elements $t_{\mathbf{T}} : \mathbf{T} \rTo \mathbf{H}$, a {\em noncommutative Lawvere topology} to be a natural transformation (of presheaves of sets)
\[
j_{\mathbf{H}} : \mathbf{H} \rTo \mathbf{H} \]
satisfying the properties

(NLT1) : $j_{\mathbf{H}} \circ t_{\mathbf{H}} = t_{\mathbf{H}}$,

(NLT2) : $j_{\mathbf{H}} \circ j_{\mathbf{H}} = j_{\mathbf{H}}$,

\[
\xymatrix{\mathbf{T} \ar[r]^{t_{\mathbf{H}}} \ar[rd]_{t_{\mathbf{H}}} & \mathbf{H} \ar[d]^{j_{\mathbf{H}}} \\ & \mathbf{H}}~\qquad
\xymatrix{\mathbf{H} \ar[r]^{j_{\mathbf{H}}} \ar[rd]_{j_{\mathbf{H}}} & \mathbf{H} \ar[d]^{j_{\mathbf{H}}} \\ & \mathbf{H}}~\qquad
\]
and where we replace the third commuting diagram by

(NLT3) : For every object $C$ in $\mathbf{C}$, every top-element $t \in \mathbf{T}(C)$ and all $x,y \in t_{\downarrow} \subset \mathbf{T}(C)$ we have
 the condition
\[
j_{\mathbf{H}}(C)(x \wedge y) = j_{\mathbf{H}}(C)(x) \wedge j_{\mathbf{H}}(C)(y) \]

\begin{lemma} A noncommutative Lawvere topology $j_{\mathbf{H}} : \mathbf{H} \rTo \mathbf{H}$  induces for every presheaf $\mathbf{P}$ a closure operator on the noncommutative complete Heyting algebra $\mathbf{Sub}_{\mathbf{H}}(\mathbf{P})$.
\end{lemma}

\begin{proof} Let $N : \mathbf{P} \rTo \mathbf{H}$ be a natural transformation and consider the inner pullback square
\[
\xymatrix{\overline{\mathbf{Q}} \ar[rrr] \ar[ddd] & & & \mathbf{T} \ar[ddd] \\
& \mathbf{Q} \ar[r] \ar[d] \ar@{.>}[lu] & \mathbf{T} \ar[ru]^{id} \ar[d] & \\
& \mathbf{P} \ar[r]^N \ar[ld]_{id} & \mathbf{H} \ar[rd]^{j_{\mathbf{H}}} & \\
\mathbf{P} \ar[rrr]_{j_{\mathbf{H}} \circ N} & & & \mathbf{H} } \]
then the composed morphism $j_{\mathbf{H}} \circ N$ gives the outer square, and hence determines an element in $\mathbf{Sub_{\mathbf{H}}}(\mathbf{P})$
\[
\overline{(\mathbf{Q},N)} = ( \overline{\mathbf{Q}}, j_{\mathbf{H}} \circ N) \]
The dashed morphism exists because the outer square is a pullback diagram, and hence we have $\mathbf{Q} \subseteq \overline{\mathbf{Q}}$ and therefore
\[
(\mathbf{Q},N) \leq \overline{(\mathbf{Q},N)} \quad \text{and} \quad \overline{\overline{(\mathbf{Q},N)}} = \overline{(\mathbf{Q},N)} \]
where the latter follows from $j_{\mathbf{H}} \circ j_{\mathbf{H}} = j_{\mathbf{H}}$ .
\end{proof}

Recall that a {\em Grothendieck topology} on $\mathbf{C}$, see for example \cite[III.\S 2]{MM}, is a function $J$ which assigns to each object $C$ a collection $J(C)$ of sieves on $C$, satisfying the following requirements

(GT1): the maximal sieve $\mathbf{y}(C) = \{ f~|~codom(f)=C~\} \in J(C)$;

(GT2): if $S \in J(C)$, then $h^*(C) \in J(D)$ for all arrows $h : D \rTo C$,

(GT3): if $R$ is a sieve on $C$ such that $h^*(R) \in J(D)$ for all $h : D \rTo C \in S \in J(C)$, then $R \in J(C)$.

\vskip 3mm

If $\mathbf{C}$ is a small category, Lawvere-Tierney topologies on $\widehat{\mathbf{C}}$ are in one-to-one correspondence with Grothendieck topologies on $\mathbf{C}$, see for example \cite[Thm.\ V.4.1]{MM}. One recovers the collection $J(C)$ from a Lawvere-Tierney topology $j$ as the set of all sieves $S$ on $C$ such that $j(S) = \mathbf{y}(C)$ in $\mathbf{\Omega}(C)$.

\vskip 3mm

Let us specify the construction of $\mathbf{Sub}_{\mathbf{H}}(\mathbf{P})$ for the presheaf $\mathbf{P}=\mathbf{y}C$ determined by
\[
\mathbf{y}C~:~\mathbf{C} \rTo \mathbf{Sets} \qquad D \mapsto Hom_{\mathbf{C}}(D,C) \]
Note that the subobjects of $\mathbf{y}C$ are exactly the sieves $S$ on $C$ and that by Yoneda's lemma every natural transformation $N : \mathbf{y}C \rTo  \mathbf{H}$ determines (and is determined by) $x = N(C)(id_C) \in \mathbf{H}(C)$. Conversely, every element $x \in \mathbf{H}(C)$ determines the pull-back diagram
\[
\xymatrix{S \ar[rr] \ar[d] & & \mathbf{T} \ar[d]^{t_{\mathbf{H}}} \\ \mathbf{y}C \ar[rr]_x & & \mathbf{H}} \]
where $S$ is the sieve on $C$ specified by
\[
S = \{ D \rTo^f C~:~\mathbf{H}(f)(x) \in \mathbf{T}(D) \} \]
Observe that $S$ is indeed a sieve as the maps $\mathbf{H}(g)$ for $E \rTo^g D$ induce a map on the top-elements $\mathbf{T}(D) \rTo \mathbf{T}(E)$. Therefore,
\[
\mathbf{Sub}_{\mathbf{H}}(\mathbf{y}C) = \{ (S,x) \in \mathbf{\Omega}(C) \times \mathbf{H}(C)~|~S= \{ D \rTo^f C~:~\mathbf{H}(f)(x) \in \mathbf{T}(D) \} \} \]
We have seen that $\mathbf{Sub}_{\mathbf{H}}(\mathbf{y}C)$ is a noncommutative complete Heyting algebra, having as its set of top-elements
\[
T(\mathbf{Sub}_{\mathbf{H}}(\mathbf{y}C)) = \{ (\mathbf{y}(C),t)~|~t \in \mathbf{T}(C) \} \]
and with minimal element $(\emptyset,0)$. If $j_{\mathbf{H}} : \mathbf{H} \rTo \mathbf{H}$ is a noncommutative Lawvere topology, the corresponding closure operation on $\mathbf{Sub}_{\mathbf{H}}(\mathbf{y}C)$ can be specified as
\[
\overline{(S,x)} = (\overline{S},j_{\mathbf{H}}(C)(x)) \quad \text{with} \quad \overline{S} = \{ D \rTo^f C~:~\mathbf{T}(f)(j_{\mathbf{H}}(C)(x)) \in \mathbf{T}(D) \} \]
Motivated by the above correspondence between Lawvere-Tierney and Grothendieck topologies, we can now define:

\begin{definition} Let $j_{\mathbf{H}} : \mathbf{H} \rTo \mathbf{H}$ be a noncommutative Lawvere topology, then the corresponding {\em noncommutative Grothendieck topology} $J_{\mathbf{H}}$ assigns to every object $C$ of $\mathbf{C}$ the collection of elements from $\mathbf{Sub}_{\mathbf{H}}(\mathbf{y}C)$
\[
J_{\mathbf{H}}(C) = \{ (S,x) \in \mathbf{\Omega}(C) \times \mathbf{H}(C)~|~(S,x) \in \mathbf{Sub}_{\mathbf{H}}(\mathbf{y}C)~\text{and}~ j_{\mathbf{H}}(C)(x) \in \mathbf{T}(C) \} \]
\end{definition}

If $J$ is a Grothendieck topology on $\mathbf{C}$ then a presheaf $\mathbf{P}$ of sets on $\mathbf{C}$ is called a sheaf for $J$ if and only if for every object $C$ of $\mathbf{C}$, every sieve $S \in J(C)$ (considered as a subobject of $\mathbf{y}C$) and every natural transformation $g : S \rTo \mathbf{P}$, there is a unique natural transformation $\mathbf{y}C \rTo \mathbf{P}$ making the diagram below commute
\[
\xymatrix{& \mathbf{y}C \ar@{.>}[rd]^{\exists !} & \\ S \ar[ru] \ar[rr]^g \ar[rd] & & \mathbf{P} \ar[ld]  \\ & \mathbf{1}} \]
Clearly, the canonical bottom maps to the terminal object $\mathbf{1}$ are superfluous in the definition, but they may help to motivate the definition below.

\vskip 3mm

Let $\mathbf{H}$ be a noncommutative subobject generator with presheaf of top-elements $\mathbf{T}$ and let $j_{\mathbf{H}} : \mathbf{H} \rTo \mathbf{H}$ be a noncommutative Lawvere topology, then the corresponding noncommutative Grothendieck topology $J_{\mathbf{H}}$ assigns to every object $C$ a collection $J_{\mathbf{H}}(C)$ of couples $(S,x)$ where $S$ is a subobject of $\mathbf{y}C$ and $x : S \rTo \mathbf{T}$ is a natural transformation which is the restriction to $S$ of a natural transformation $x : \mathbf{y}C \rTo \mathbf{H}$ determined by $x \in \mathbf{H}(C)$.

So, instead of the canonical morphism $S \rTo \mathbf{1}$ we have to consider certain morphisms $x : S \rTo \mathbf{T}$. Therefore it makes sense to define the category of all presheaves with respect to the noncommutative Grothendieck topology $J_{\mathbf{H}}$ to be the slice category $\widehat{\mathbf{C}}/\mathbf{T}$. That is, the objects are pairs $(\mathbf{F},\pi_{\mathbf{F}})$ with $\mathbf{F} \in \widehat{\mathbf{C}}$ and $\pi_{\mathbf{F}}$ a natural transformation $\mathbf{F} \rTo \mathbf{T}$, and morphisms compatible natural transformations $g$
\[
\xymatrix{\mathbf{F} \ar[d]_{\pi_{\mathbf{F}}} \\ \mathbf{T}} \qquad \xymatrix{\mathbf{F} \ar[rd]_{\pi_{\mathbf{F}}} \ar[rr]^g & & \mathbf{G} \ar[ld]^{\pi_{\mathbf{G}}} \\ & \mathbf{T} &} \]

\begin{definition} A presheaf $(\mathbf{F},\pi_{\mathbf{F}})$ is a sheaf with respect to the noncommutative Grothendieck topology $J_{\mathbf{H}}$ if and only if for every object $C$ of $\mathbf{C}$, every element $(S,x) \in J_{\mathbf{H}}(C)$, and every morphism $g$ in $\widehat{\mathbf{C}}/\mathbf{T}$
\[
\xymatrix{ & \mathbf{y} C \ar@{.>}[rd]^{\exists !} & \\ S \ar[rr]^g \ar[rd]_x \ar[ru] & & \mathbf{F} \ar[ld]^{\pi_{\mathbf{F}}} \\
& \mathbf{T} & } \]
there is a unique morphism $\mathbf{y} C \rTo \mathbf{F}$ in $\widehat{\mathbf{C}}$. Here $S \rTo^x \mathbf{T}$ is the pull-back map induced by the natural transformation $x : \mathbf{y} C \rTo \mathbf{H}$.

The {\em noncommutative topos} $\mathbf{Sh}(\mathbf{C},J_{\mathbf{H}})$ has as its objects all sheaves with respect to the noncommutative Grothendieck topology $J_{\mathbf{H}}$ and morphisms as in $\widehat{\mathbf{C}}/\mathbf{T}$.
\end{definition}

\section{A class of examples}

In this section we will construct examples of noncommutative subobject classifiers and show that a noncommutative topos does not have to be an elementary topos.

\vskip 3mm

First, we will construct complete noncommutative  Heyting algebras. By a result of \cite{KarinNCH} complete noncommutative  Heyting algebras are exactly noncommutative frames (together with a distinguished element in the top $\mathcal D$-class), where a \emph{noncommutative frame} is a strongly distributive, join complete skew lattice that satisfies the infinite distributive laws. 

Let $h$ be a (commutative) complete Heyting algebra. 
Since $h$ is a distributive lattice it embeds into $\prod_{i\in I} \mathbf 2$ for some index set $I$, where $\mathbf 2$ is the two element lattice
\[
\mathbf{2} = \xymatrix{1 \ar@{-}[d] \\ 0} \qquad \text{and define} \qquad \widehat{P} = \xymatrix{& p \ar@{.}[l] \ar@{.}[r] \ar@{-}[d] & \\ & 0 &} \]
to be the skew lattice on $\widehat{P} = \{ 0 \} \cup P$, with a unique bottom element $0$ and a set $P$ of top elements, and operations are defined by:
\[x,y\in P : x\land y=x, \qquad x\lor y=y.
\]
\[x\land 0=0=0\land x, \qquad x\lor 0=x=0\lor x,
\]
Note that $\widehat{P}$ is a strongly distributive skew lattice and has two $\mathcal D$-classes:  bottom class $\{0\}$ and  top class $P$, whence $\widehat{P}/\mathcal{D} \simeq \mathbf{2}$.

Let $H$ be the pullback (in $\mathbf{Sets}$) of the following diagram: 
\[
\xymatrix{H \ar[rr] \ar[d] & & \prod_{i\in I} \widehat{P} \ar[d]^{/\mathcal D} \\ 
h\ar[rr]^i & & \prod_{i\in I} \mathbf 2 }
\]
Denoting by $\pi_i$ the projection to the $i$-th factor we obtain a commutative diagram:
\begin{equation}\label{eq:diag-lemma6}
\xymatrix{H \ar[rr]^{\pi_i} \ar[d] & &  \widehat{P} \ar[d]^{/\mathcal D} \\ 
h\ar[rr] & & \mathbf 2 }
\end{equation}

\begin{lemma} With notations as above, $H$ becomes a noncommutative frame with bottom $0$ and top $\mathcal{D}$-class $T(H) = \prod_{i\in I} P$ under the operations
\[
(x_i)_i\wedge (y_i)_i= (x_i \wedge y_i)_i\qquad \text{and} \qquad (x_i)_i \vee (y_i)_i = (x_i \vee y_i)_i \]
where the bracketed operations are performed in the skew lattice $\widehat{P}$. In particular, $H/\mathcal{D} \simeq h$. If we fix a distinguished element $d\in H$ s.t.  $\pi_i(d)\neq 0$ for all $i\in I$
then $H$ is a complete noncommutative Heyting algebra.
\end{lemma}

\begin{proof} 
First we observe that $H$ is a strongly distributive skew lattice because it embeds into a power of $\widehat{P}$ and strongly distributive skew lattices form a variety.  Note that elements $x,y\in H$ are $\mathcal D$-equivalent exactly when for all $i\in I$: ($\pi_i(x)=0$ iff $\pi_i(y)=0$). A commuting subset in $H$ is of the form $\{x_j\,|\, j\in J\}$ s.t. $\pi_i(x_j)\neq 0$ together with $\pi_i(x_k)\neq 0$ implies $\pi_i(x_j)=\pi_i(x_k)$, for all $j,k\in J $ and all $i\in I$.  
Skew lattice $H$ is join complete because $h$ is complete and the diagram \eqref{eq:diag-lemma6} commutes. It remains to prove that $H$ satisfies the infinite distributive laws. Given a commuting subset $\{x_j\}\subseteq H$, $y\in H$ and $i\in I$ we need to show that:
\[\pi_i(\bigvee x_j \land y)=\pi_i(\bigvee (x_j\land y)) \text{ and } \pi_i(y\land\bigvee x_j )=\pi_i(\bigvee (y\land x_j)) 
\]
First we observe that $\{x_j\land y\,|\, j\in J\}$ and $\{y\land x_j\,|\,j\in J\}$ are again  commuting subsets. 
%Let $i\in I$ and assume 
%$\pi_i(x_j)\neq 0$, $\pi_i(x_k)\neq 0$. Then $\pi_i(x_j\land y)=\pi_i(x_j)\land \pi_i(y)=\pi_i(x_k)\land \pi_i(y)=\pi_i(x_k\land y)$.  
Note that if $\pi_i(x_j)\neq 0$ for some $j$ then $\pi_i(\bigvee x_j \land y)=\pi(x_j \land y)=\pi_i(\bigvee (x_j\land y))$. If $\pi_i(x_j)=0$ for all $j$ then $\pi_i(\bigvee x_j \land y)=0=\pi_i(\bigvee (x_j\land y))$.
\end{proof}

\begin{lemma} \label{example} For every contravariant functor
\[
\mathbf{h}~:~\mathbf{C} \rTo \mathbf{cHA} \]
and every presheaf $\mathbf{P} \in \widehat{\mathbf{C}}$ with a global section $d : \mathbf{1} \rTo \mathbf{P}$ there is a contravariant functor
\[
\mathbf{H}~:~\mathbf{C} \rTo \mathbf{ncHA} \qquad C \mapsto \mathbf{H}(C) \]
where $\mathbf{H}(C)$ is the complete noncommutative Heyting algebra constructed in the previous lemma from the complete Heyting algebra $h = \mathbf{h}(C)$ and the set $P= \mathbf{P}(C)$, with presheaf of top elements $\mathbf{T}$. Moreover, $\mathbf{H} /\mathcal{D} \simeq \mathbf{h}$.

In the special case when $\mathbf{h} = \mathbf{\Omega}$ we obtain for every presheaf $\mathbf{P}$ with a global section a noncommutative subobject classifier $\mathbf{H}$ with $\mathbf{H}/\mathcal{D} \simeq \mathbf{\Omega}$.
\end{lemma}

\begin{proof} Follows immediately from  the previous lemma.
\end{proof}

Let us work out an explicit example. Let $\mathbf{C}$ be the category having two objects $V$ and $E$ and two non-identity morphisms $s,t : V \rTo E$, then it is easy to see that the presheaf topos
\[
\widehat{\mathbf{C}} \simeq \mathbf{diGraph} \]
is the category of directed graphs. A presheaf $\mathbf{P} : \mathbf{C} \rTo \mathbf{Sets}$ determines a set of vertices $\mathbf{P}(V)$ and edges $\mathbf{P}(E)$ and the two maps $\mathbf{P}(s),\mathbf{P}(t) : \mathbf{P}(E) \rTo \mathbf{P}(V)$ assign to an edge its starting resp. terminating vertex. The subobject classifier $\mathbf{\Omega}$ is given by
\[
\begin{cases}
\mathbf{\Omega}(E) = \{ 1 = \{ id_E,s,t \}, U=\{ s,t \}, S = \{ s \}, T=\{ t \}, 0 = \emptyset \} \\
\mathbf{\Omega}(V) = \{ 1 = \{ id_V \}, 0 = \emptyset \}
\end{cases}
\]
and corresponds to the directed graph
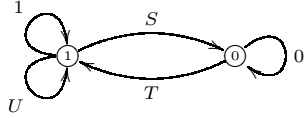
\[
\xymatrix{\vtx{1} \ar@(l,u)^{1} \ar@(l,d)_{U} \ar@/^2ex/[rr]^S & & \vtx{0} \ar@(ur,dr)^{0} \ar@/^2ex/[ll]^{T}} \]
with the terminal subobject $\mathbf{1}$ corresponding to the subgraph on the loop $1$. The Heyting algebras have poset structure

\vskip 3mm

\[
\mathbf{\Omega}(E) = \quad \xymatrix{& 1 \ar@{-}[d] & \\ & U \ar@{-}[ld] \ar@{-}[rd] & \\ S \ar@{-}[rd] & & T \ar@{-}[ld] \\ & 0 &} \qquad  \qquad  \mathbf{\Omega}(V)= \quad \xymatrix{1 \ar@{-}[d] \\ 0} \]
It is easy to verify that there are exactly $4$ Lawvere-Tierney topologies on $\widehat{\mathbf{C}}$ with corresponding Grothendieck topologies on $\mathbf{C}$ and corresponding sheafifications:
\begin{enumerate}
\item{$J_1(V)=\{ 1 \}$ and $J_1(E)= \{ 1 \}$, the chaotic topology.  All presheaves are $J_1$-sheaves and the sheafification functor is the identity.}
\item{$J_2(V)=\{ 1 \}$ and $J_2(E)=\{ 1,U \}$. The sheaf condition for $\mathbf{P}$ asserts that for all $v,w \in \mathbf{P}(V)$ there is a unique edge $e$ with $s(e)=v$ and $t(e)=w$. That is, sheaves are the complete directed graphs, and the sheafification of a directed graph is the complete directed graph on the vertices.}
\item{$J_3(V)=\{ 1,0 \}$ and $J_3(E)=\{ 1 \}$. The only non-maximal covering sieve on $V$ is the empty sieve. A presheaf $\mathbf{P}$ is a $J_3$-sheaf if and only if $\mathbf{P}(V)$ is a singleton. The sheafification sends the vertices of a directed graph all to the same vertex and each edge to a different loop. }
\item{$J_4(V)=\{ 1,0 \}$ and $J_4(E)=\{ 1,U,S,T,0 \}$, the discrete topology. Here the only sheaf is the terminal object (a one loop graph) and sheafification is the unique map to the terminal object.}
\end{enumerate}

\vskip 3mm

Consider the presheaf $\mathbf{P} = 
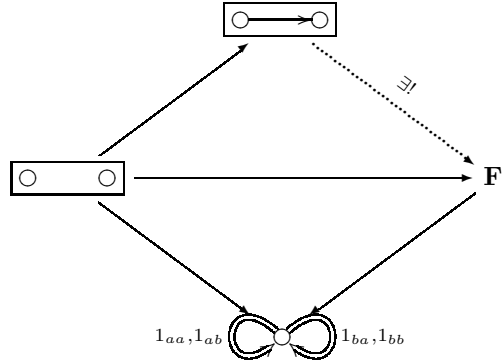
\xymatrix{\vtx{x} \ar@(ul,dl)_{a} \ar@(ur,dr)^{b}}$, then the noncommutative subobject classifier $\mathbf{H}$ corresponding to $\mathbf{\Omega}$ and $\mathbf{P}$ as constructed in lemma~\ref{example} can be slightly simplified such that $\mathbf{H}(E)$ has only $4$ top elements, rather than the $8$ given by the construction. The corresponding directed graph is

\vskip 2mm

\[
\xymatrix{\vtx{x}  \ar@2@(ur,ul)_{1_{aa},1_{ab}} \ar@2@(ul,l)_{1_{ba},1_{bb}} \ar@2@(l,dl)_{U_{aa},U_{ab}} \ar@2@(dl,dr)_{U_{ba},U_{bb}} \ar@/^2ex/[rrr]^{S_a} \ar@/^6ex/[rrr]^{S_b} & & & \vtx{0} \ar@/^2ex/[lll]^{T_a} \ar@/^6ex/[lll]^{T_b} \ar@(ur,dr)^{0}} \]

\vskip 3mm

\noindent
with the subobject $\mathbf{T} \rTo \mathbf{H}$ corresponding to the subgraph on the $4$ loops $1_{aa},1_{ab},1_{ba}$ and $1_{bb}$. The poset structure on the noncommutative Heyting algebras is $\mathbf{H}(V) \simeq \mathbf{\Omega}(V) \simeq \mathbf{2}$ and

\vskip 3mm
\[
\mathbf{H}(E) = \xymatrix{ & 1_{aa} \ar@{-}[d] \ar@{.}[r] & 1_{ab} \ar@{-}[d] \ar@{.}[rr] & & 1_{ba} \ar@{-}[d] \ar@{.}[r] &1_{bb} \ar@{-}[d] & \\
& U_{aa} \ar@{.}[r] \ar@{-}[d] \ar@{-}[rrrd] & U_{ab} \ar@{.}[rr] \ar@{-}[ld] \ar@{-}[rrrd] & & U_{ba} \ar@{.}[r] \ar@{-}[d] \ar@{-}[lld] & U_{bb} \ar@{-}[llld] \ar@{-}[d] & \\
& S_a \ar@{.}[r] \ar@{-}[rrd] & S_b \ar@{-}[rd]  &  &  T_a \ar@{-}[ld] \ar@{.}[r]& T_b \ar@{-}[lld] & \\
& & & 0 & & & }   \]
We will next determine the noncommutative toposes determined by noncommutative Grothendieck topologies associated to $\mathbf{H}$. The category of presheaves is the slice category $\widehat{\mathbf{C}}/\mathbf{T}$. A directed graph with a morphism $\pi_{\mathbf{F}} : \mathbf{F} \rTo \mathbf{T}$ is a directed graph with a $4$-coloring of its edges. Morphisms in $\widehat{\mathbf{C}}/\mathbf{T}$ are directed graph morphisms preserving the coloring of edges.

\begin{lemma} There are exactly $16$ noncommutative Grothendieck topologies associated to $\mathbf{H}$:
\[
j_{\mathbf{H}}(E) = \{ 1_{aa},1_{ab},1_{ba},1_{bb} \} \cup S  \quad \text{with} \quad S \subseteq \{ U_{aa},U_{ab},U_{ba},U_{bb} \} \]
Any $4$-colored digraph satisfies the sheaf condition if $S = \emptyset$. For the noncommutative Grothendieck topologies with $S \not= \emptyset$ the sheaves are exactly the complete digraphs with a $4$-coloring. 
\end{lemma}

\begin{proof} Assume that a noncommutative Lawvere topology $j_{\mathbf{H}} : \mathbf{H} \rTo \mathbf{H}$ is such that $j_{\mathbf{H}}(V)(0) = 1$, then $j_{\mathbf{H}}(E)(0) \in \{ 1_{aa},1_{ab},1_{ba},1_{bb},U_{aa},U_{ab},U_{ba},U_{bb} \}$ which is impossible because $j_{\mathbf{H}}(E)$ must be order preserving. Therefore $j_{\mathbf{H}}(V) = id_{\mathbf{H}(V)}$. As a consequence the Grothendieck topologies on $\mathbf{C}$ corresponding to the global sections $1_{aa},1_{ab},1_{ba},1_{bb}$ can only be either $J_1$ or $J_2$, giving the 16 cases. If $S = \emptyset$ we have no conditions to satisfy for $\mathbf{F} \rTo \mathbf{T}$.

If, however $S \not= \emptyset$, each occurrence of $U_{aa},U_{ab},U_{ba}$ or $U_{bb}$ gives rise to a condition

\vskip 2mm

\[
\begin{diagram}
& & \boxed{\xymatrix{\vtx{} \ar[r] & \vtx{}}} &  & \\
& \ruTo & & \rdDotsto^{\exists !} & \\
\boxed{\xymatrix{\vtx{} & \vtx{}}} &  & \rTo & & \mathbf{F} \\
& \rdTo & & \ldTo & \\
& & \xymatrix{\vtx{} \ar@2@(ul,dl)_{1_{aa},1_{ab}} \ar@2@(ur,dr)^{1_{ba},1_{bb}}} & &
\end{diagram}
\]

\vskip 2mm

\noindent
which means that for every pair of vertices $v,w \in \mathbf{F}(V)$ there must be a unique edge $\xymatrix{\vtx{v} \ar[r] & \vtx{w}}$. Note that the color of this unique edge is not imposed by $U_{aa},U_{ab},U_{ba}$ or $U_{bb}$. Therefore, $\mathbf{F}$ is a sheaf for the noncommutative Grothendieck topology if and only if $\mathbf{F}$ is a complete digraph with a certain $4$-coloring of the edges determined by the map $\mathbf{F} \rTo \mathbf{T}$.
\end{proof}

It does follow that for any noncommutative Grothendieck topology $J_{\mathbf{H}}$ with $S \not= \emptyset$  the noncommutative topos $\mathbf{Sh}(\mathbf{C},J_{\mathbf{H}})$ is {\em not} a Grothendieck topos, nor even an elementary topos, as it fails to have a terminal object (the four loop graph with one loop of each color is {\em not} a sheaf).

\end{document}